\documentclass[11pt]{amsart}
\usepackage{graphicx}
\usepackage[mathcal]{euscript}
\usepackage{float}
\usepackage{cite}
\usepackage{verbatim}
\usepackage{url}
\usepackage{tikz}
\usepackage{adjustbox}
\usepackage{amsmath}
\usepackage{multicol}
\usepackage{rotating}
\usepackage{enumitem}
\usepackage{subcaption}
\usepackage{verbatim}

\title{On Khovanov homology and related invariants}

\usepackage{graphicx}
\usepackage{amsmath, amssymb}
\usepackage{hyperref}

\hypersetup{
    colorlinks=true,
    linkcolor=red,
    filecolor=magenta,      
    urlcolor=red,
}
\definecolor{mzpink}{RGB}{219, 48, 122}
\definecolor{blue(pigment)}{rgb}{0.2, 0.2, 0.6}

\definecolor{aog}{rgb}{0.0, 0.5, 0.0}

\newcommand{\R}{\mathbb{R}}

\newtheorem{theorem}{Theorem}
\newtheorem{lemma}[theorem]{Lemma}

\newtheorem*{corollaryconj*}{Corollary of Conjecture \ref{conjecture:Signature}}
\newtheorem{proposition}[theorem]{Proposition}

\theoremstyle{definition}

\newtheorem{example}[theorem]{Example}

\newcommand{\pLee}{p_{\textnormal{Lee}}}

\newcommand{\dthin}{d_{\textnormal{thin}}}

\def\alt{\operatorname{alt}}

\def\pLee{\operatorname{pg}_{\textnormal{Lee}}}
\def\pBN{\operatorname{pg}_{\textnormal{BN}}}
\def\dthin{d_{\textnormal{thin}}}

\def\dlee{d_{\textnormal{Lee}}}

\def\dBN{d_{\textnormal{BN}}}

\usetikzlibrary{decorations.pathreplacing}
\usetikzlibrary{arrows.meta}

\usetikzlibrary{arrows, shapes, decorations, decorations.markings, backgrounds, patterns, hobby, knots, calc, positioning}

\tikzset{invclip/.style={clip,insert path={{[reset cm]
      (-16383.99999pt,-16383.99999pt) rectangle (16383.99999pt,16383.99999pt)
    }}}}

\author[C. Caprau]{Carmen Caprau}
\thanks{CC was partially supported by Simons Foundation Collaboration Grant $355640$}
\address{Department of Mathematics, California State University, Fresno, CA 93740}
\email{ccaprau@csufresno.edu}

\author[N. Gonz\'{a}lez]{Nicolle  Gonz\'{a}lez}
\address{UCLA Department of Mathematics, 520 Portola Plaza, Los Angeles, CA 90095}
\email{nicolle@math.ucla.edu}

\author[C. Lee]{Christine Ruey Shan Lee}
\thanks{CL is supported in part by NSF Grant DMS 1907010}
\address{Department of Mathematics and Statistics, University of South Alabama, Mobile AL 36608}
\email{crslee@southalabama.edu}

\author[A. Lowrance]{Adam M. Lowrance}
\thanks{AL was supported in part by NSF Grant DMS 1811344.}
\address{Department of Mathematics and Statistics,
Vassar College, Poughkeepsie, NY 12604}
\email{adlowrance@vassar.edu}

\author[R. Sazdanovi\'{c}]{Radmila Sazdanovi\'{c}}
\thanks{RS partially supported by the Simons Foundation Collaboration Grant 318086  and NSF Grant DMS 1854705.}
\address{Department of Mathematics, North Carolina State University, Raleigh, NC 27695}
\email{rsazdan@ncsu.edu}

\author[M.\ Zhang]{Melissa Zhang}
\address {Department of Mathematics, University of Georgia, Athens, GA 30602}
\email{melissa.zhang@uga.edu}

\begin{document}

\maketitle

\begin{abstract}
This paper begins with a survey of some applications of Khovanov homology to low-dimensional topology, with an eye toward extending these results to $\mathfrak{sl}(n)$ homologies. We extend Levine and Zemke's ribbon concordance obstruction from Khovanov homology to $\mathfrak{sl}(n)$ homology for $n \geq 2$, including the universal $\mathfrak{sl}(2)$ and $\mathfrak{sl}(3)$  homology theories. Inspired by Alishahi and Dowlin's bounds for the unknotting number coming from Khovanov homology and relying on spectral sequence arguments, we produce bounds on the alternation number of a knot. Lee and Bar-Natan spectral sequences also provide lower bounds on Turaev genus.
\end{abstract}

\section{Introduction}

The discovery of the Jones polynomial \cite{Jones} has invigorated low-dimensional topology by introducing a plethora of link and 3-manifold invariants. Efforts to study these quantum invariants have yielded powerful new link invariants, in the form of homology theories, through categorification. In this article, we focus on the impact of the most influential homology theory arising from quantum invariants: Khovanov homology \cite{Kh00}. Our goal is to sample some recent applications of Khovanov-type theories to smooth low-dimensional topology. By bringing together the various ideas and constructions, we hope to facilitate new applications.

In Section~\ref{s.survey}, we curate a survey of recent developments in knot concordance, mutation detection, unknotting, and the categorification of knot polynomials.
Note that our overview will focus on Lee's spectral sequence and Rasmussen's $s$-invariant, and generalizations of these constructions.
We will exclude results linking Khovanov homology to knot Floer homology or Heegaard Floer homology, for which the readers may consult the resources \cite{OS05}, \cite{Baldwin}, \cite{LoZe}. We also exclude applications toward low-dimensional contact and symplectic geometry.

Following the survey, we give two new applications. 
In Section~\ref{s.homologyconcordance}, we extend Levine-Zemke's \cite{LZ} ribbon concordance obstruction from Khovanov homology to $\mathfrak{sl}(n)$ homology for $n \geq 2$, as well as to universal $\mathfrak{sl}(2)$ and $\mathfrak{sl}(3)$ homology theories. More generally, we show that a ribbon concordance between links induces injective maps on link homologies defined via webs and foams modulo relations. Kang provides a different approach in \cite{Kang}, where it is shown that a ribbon concordance induces injective maps on link homology theories that are multiplicative link TQFTs and which are either associative or Khovanov-like. Our proof relies mainly on the fact that all of the homology theories considered in Section~\ref{s.homologyconcordance} satisfy certain cutting neck and sphere relations in the category of dotted cobordisms, without the need to provide new definitions or develop special techniques.

In Section~\ref{section:Spectral}, we  use spectral sequences coming from Khovanov homology to bound the alternation number, as well as the Turaev genus of a knot in $S^3$.

We hope that this article provides a convenient reference to those entering this area of research and sparks interest in the subject.

\subsection*{Acknowledgements}
We would like to thank ICERM and the organizers of the Women in Symplectic and Contact Geometry and Topology workshop (WiSCon) for the opportunity to work on this project during the summer of 2019.

\section{A survey of applications of Khovanov homology}
\label{s.survey}

\subsection{Rasmussen's $s$-invariant}
Possibly the most well-known application of the original Khovanov homology \cite{Kh00} lies in Rasmussen's \cite{Ras10} concordance invariant $s$, which comes from a spectral sequence arising from a filtration on the Lee complex. The Lee spectral sequence is a key ingredient of the proof that the Khovanov homology of alternating knots is thin \cite{Lee:Endo}.
Rasmussen shows that $s$ induces a homomorphism from the concordance group to the integers. Therefore, it provides a slice obstruction. 
In fact, $s$ gives lower bounds on the slice genus of a knot. As an example of an application, he uses this to give a strikingly short proof of the Milnor conjecture, which was previously proven by Kronheimer and Mrowka using gauge theory \cite{KM-Khovanov-unknot}. 

Many others have since modeled the algebraic construction of Rasmussen's invariant to produce more concordance invariants, many of which are generalizations of the $s$-invariant to $\mathfrak{sl}(n)$ homology (\cite{Lobb}, \cite{Wu}, \cite{Lobb2}, \cite{Lewark}) or to the universal $\mathfrak{sl}(2)$ homology \cite{Caprau2012}. 

In 2012, Lipshitz and Sarkar introduced a stable homotopy type for Khovanov homology \cite{LS}. In \cite{LS-Sq}, they define a refinement of $s$ for each stable cohomology operation, and show that the refinement corresponding to $Sq^2$ is in general stronger than $s$ (see also \cite{Seed-Sq}).

\subsection{Mutants}

Mutant knots are notoriously difficult to distinguish using knot invariants. It has been shown that for a knot, mutation preserves the signature, the Alexander polynomial, the volume (if the mutants in question are hyperbolic), and the Jones polynomials \cite{CL99}, \cite {Mor88}, \cite{Rub87}. It is an open question whether Khovanov homology is invariant under mutation on knots. While there exist mutant \emph{links} with distinct Khovanov homologies (see \cite{wehrli-03}), it has been shown that odd Khovanov homology and Khovanov homology with $\mathbb{F}_2$ coefficients are invariant under mutation; for details, we refer the reader to \cite{Blo10} and \cite{Weh10}, respectively.

There has been some recent indication that Khovanov-type theories may be used to distinguish mutants. For example, a prominent open problem was resolved when Piccirillo showed that the Conway knot is not slice, using the $s$-invariant \cite{Piccirillo} defined by Rasmussen from the spectral sequence from Khovanov homology to Lee homology. Lobb-Watson's \cite{Lobb-Watson-involution} filtered invariant is able to detect mutants in the presence of an involution. In a different direction, one may also consider generalized mutations along genus 2 surfaces from which  (Conway) mutation may be recovered. It has been shown that Khovanov homology distinguishes a pair of generalized mutants, while the signature, HOMFLY-PT polynomial, Jones polynomials, and Kauffman polynomial are the same \cite{DGST10}.

\subsection{Ribbon Concordance}
Motivated by Gordon's conjecture~\cite{Gordon-ribbon} that ribbon concordance gives a partial ordering on knots in $S^3$, there has been great interest in studying the behavior of knot invariants under ribbon concordance. Notably, in 2019, Zemke \cite{Zemke-ribbon-HFK} showed that knot Floer homology obstructs ribbon concordance. This led to an exciting series of papers extending this result to various homology-type invariants for knots. Within the realm of Khovanov-type invariants, Levine-Zemke~\cite{LZ} extended the result to the original Khovanov homology, Kang~\cite{Kang} extended the result to a setup that includes Khovanov-Rozansky homologies~\cite{KR08}, knot Floer homologies and other theories, and Sarkar~\cite{Sarkar-ribbon} defined the notion of ribbon distance and derived bounds on this from Khovanov-Lee homology.

\subsection{Unknotting and unlinking via spectral sequences} Besides the $s$-invariant and its relationship to the slice genus, one can also relate spectral sequences from Khovanov homology to other link invariants. Alishahi and Dowlin \cite{AD:Lee} proved that the page at which the Lee spectral sequence collapses can be used to give a lower bound on the unknotting number of the knot. A consequence of this bound is that the Knight Move conjecture holds for all knots with unknotting number at most two. Alishahi also proved a similar lower bound for the unknotting number using the Bar-Natan spectral sequence coming from the characteristic two Khovanov homology \cite{Alishahi:BN}. In another direction, Batson and Seed \cite{BS:LinkSplit} constructed a spectral sequence starting with the Khovanov homology of a link and converging to the Khovanov homology of the disjoint union of its components. The page at which this spectral sequence collapses yields a lower bound on the link splitting number of the link.

\subsection{$\mathfrak{sl}(n)$ homology and HOMFLY-PT homology}
For each $n$, the $\mathfrak{sl}(n)$ link invariant is a certain one-variable specialization of the HOMFLY-PT polynomial. In \cite{KR08}, Khovanov and Rozansky gave a categorifiction of the $\mathfrak{sl}(n)$ polynomial using matrix factorizations. Moreover, using matrix factorizations with a different potential, Khovanov and Rozansky \cite{KR082} constructed a categorification of the HOMFLY-PT polynomial. For the $\mathfrak{sl}(3)$ link invariant, Khovanov \cite{Kh04} constructed another categorification using trivalent webs and foams between such webs. This was later generalized to the universal $\mathfrak{sl}(3)$ homology by Mackaay and Vaz \cite{MV07}. An approach to the universal $\mathfrak{sl}(2)$ homology theory was constructed by Caprau  \cite{Caprau09}, using a combination of ideas from \cite{BN05} and \cite{Kh04}.  In \cite{MSV}, Mackaay, Stosic, and Vaz gave a topological categorification of the $\mathfrak{sl}(n)$ polynomial, for all $n\geq 4,$ via webs and a special type of foams. For specific details on the versions of $\mathfrak{sl}(n)$ homologies that are used in this paper, we refer the reader to Section \ref{s.homologyconcordance}. A potential topological application of $\mathfrak{sl}(n)$ and HOMFLY-PT homologies is that  they would be better able to distinguish mutant knots, due in part to the fact that the corresponding decategorifications can detect mutants \cite{Morton}, \cite{MortonRyder}. 

\section{Link homologies and ribbon concordance} \label{s.homologyconcordance}

Let $L_0$ and $L_1$ be links in $S^3$. A \textit{concordance} $C \subset S^3\times [0,1]$ from $L_0$ to $L_1$ is an embedding $f: S^1 \times [0,1] \to S^3 \times [0, 1]$ such that $f(S^1 \times \{0\}) = L_0 \times \{0\}$ and $f(S^1 \times \{1\}) = L_1 \times \{1\}$. In this case, we say that the links $L_0$ and $L_1$ are \textit{concordant}.

By a small isotopy of $S^3\times [0,1]$, the concordance $C$ may be adjusted so that the restriction to $C$ of the projection $S^3\times [0,1] \to [0,1]$ is a Morse function.
If this Morse function has only critical points of index $0$ (local minima) and $1$ (saddle points) (that is, if it has no critical points of index $2$ (local maxima)), then $C$ is called a \textit{ribbon concordance}. In this case, we say that $L_0$ is \textit{ribbon concordant} to $L_1$.

Denote by $\overline{C}$ the mirror image of $C$ and regard it as a concordance from $L_1$ to $L_0$. Then $\overline{C} \circ C$ is the concordance from $L_0$ to itself obtained by concatenating $C$ and $\overline{C}$. Zemke~\cite{Zemke-ribbon-HFK} proved that the concordance $\overline{C} \circ C$ can be obtained by taking the identity concordance $L_0 \times [0, 1]$ and ``tubing in'' unknotted, unlinked $2$-spheres $S_1, \dots, S_n$ using ``tubes" $T_1, \dots, T_n$.
The tubes are annuli embedded in $S^3 \times [0, 1]$, joining $L_0 \times [0, 1]$ with the spheres $S_1, \dots, S_n$. Specifically, Zemke~\cite{Zemke-ribbon-HFK} explained that the concordance $\overline{C} \circ C$ can be described, up to isotopy, by the following movie presentation:
\begin{itemize}
\item $n$ births of disjoint unknots $U_1, \dots , U_n$, each of which being disjoint from the link $L_0$;
\item $n$ saddles represented by bands $B_1,  \dots , B_n$, such that $B_i$ connects $U_i$ with $L_0$;
\item $n$ saddles represented by bands $\overline{B}_1,  \dots , \overline{B}_n$, where each $\overline{B}_i$ is respectively the mirror image (dual) of $B_i$;
\item $n$ deaths, deleting $U_1, \dots , U_n$.
\end{itemize}
The embedded annuli $T_i$ are obtained by concatenating the second and third movie frames above, that is, by joining the bands $B_i$ together with their respective dual bands, $\overline{B}_i$. The births and deaths of the unknots $U_1, \dots , U_n$ determine $n$ unknotted, unlinked $2$-spheres $S_1, \dots, S_n$. The annuli $T_i$ are the boundaries of some three-dimensional $1$-handles $h_i$, and each handle $h_i$ intersects the surface $L_0 \times [0, 1]$ and the sphere $S_i$ in some disks $D_i$ and $D_i'$, respectively.
Then, the concordance $\overline{C} \circ C$ can be thought of as the following union:
\[\overline{C} \circ C = \big ((L_0 \times [0, 1]) \diagdown (D_1 \cup \dots \cup D_n) \big )  \cup (T_1 \cup \dots \cup T_n)\cup \big ((S_1 \diagdown D_1') \cup \dots \cup (S_n \diagdown D_n')   \big ). \]

The goal of this section is to use the above result by Zemke~\cite{Zemke-ribbon-HFK} to show that a ribbon concordance between two links induces an injective map on $\mathfrak{sl}(n)$ link homologies, for all $n \geq 2$. That is, we want to show that the main result proved by Levine and Zemke in~\cite{LZ} can be generalized to universal Khovanov homology, as well as to higher rank link homologies. The proofs of the following statements are similar in nature to the proofs of the analogous statements provided in~\cite{LZ}.

Here we are considering $\mathfrak{sl}(n)$ foam homologies, which we will denote by $\mathcal{H}_n$. For $n=2$ and $n = 3$, we are working with the corresponding universal theories (`universal' in the sense as explained by Khovanov's work~\cite{Kh06}). The universal theory categorifying the $\mathfrak{sl}(2)$ link polynomial corresponds to a Frobenius system of rank two associated to the ring $\mathcal{A}_2 = \mathbb{Z}[X, h, t]/(X^2 - hX -t)$, where $h$ and $t$ are formal parameters. The homology of the unknot is the ring $\mathcal{A}_2$, and the homology of that of the empty link is the ground ring $\mathbb{Z}[X, h, t]$. To obtain a homology theory that is purely functorial with respect to link cobordisms, Caprau~\cite{Caprau09} worked with singular cobordisms and with the ground ring $\mathbb{Z}[i][X, h, t]$, where $i^2 = -1$. 
Similarly, the universal $\mathfrak{sl}(3)$ foam theory, introduced by Mackaay and Vaz in~\cite{MV07}, corresponds to a Frobenius system of rank three associated to the ring $\mathcal{A}_3 = \mathbb{Z}[X, a, b, c]/(X^3-aX^2-bX-c)$, where $a, b$, and $c$ are formal parameters. The work in~\cite{MV07} generalizes Khovanov's construction in~\cite{Kh04}. 
For $n \geq 4$, we consider the homology theory introduced by Mackaay, Sto\v{s}i\'{c}, and Vaz in~\cite{MSV}, which corresponds to the ring $\mathcal{A}_n = \mathbb{Q}[X]/(X^n)$. The foams in~\cite{MSV} are more complicated than those for the cases of $n = 2$ and $3$, as these foams have additional types of singularities and their evaluation makes use of the Kapustin-Li formula~\cite{KaLi}.

These homology theories use foams modulo local relations, as pioneered by Bar-Natan~\cite{BN05} in his approach to local Khovanov homology for tangles. In each case of the $\mathfrak{sl}(n)$ homology theory considered here for a fixed value of $n \geq 2$, one associates to a link diagram a formal chain complex in a certain abelian category $\text{Kom}(\textbf{Foam}_n)$, whose objects are column vectors of closed $1$-manifolds in the plane, and whose morphisms are matrices of dotted foams in $\mathbb{R}^2 \times [0, 1]$, which are considered up to boundary-preserving isotopies, and modulo local relations.

For our purpose, for each $\mathfrak{sl}(n)$ foam homology theory for $n \geq 2$, we will only need the local relations involving ordinary surfaces and (1+1)-cobordisms in $\mathbb{R}^2 \times [0, 1]$ marked with dots. Specifically, we will employ the sphere relations (S$_n$) and the cutting neck relation (CN$_n$), for fixed $n \geq 2$, depicted in Fig.~\ref{local relations}. In this figure, a letter $i$ on a surface means that the surface is marked with $i$ dots. Recall that in terms of the $2$-dimensional TQFT associated with the corresponding Frobenius extension and the resulting $\mathfrak{sl}(n)$ homology theory for links, a dot on a surface corresponds to the endomorphism of the ring $\mathcal{A}_n$ that is multiplication by $X$. The sphere relations (S$_n$) are the geometric counterparts of the evaluations of the counit map $\epsilon: \mathcal{A}_n \to R$ on the generators $1, X, \dots , X^{n-1}$, where $R$ is the ground ring. Moreover, the cutting neck relation (CN$_n$), for each $n\geq 2$, is the geometric representation of the formula for $\Delta(1)$, where $\Delta: \mathcal{A}_n \to \mathcal{A}_n \otimes_R \mathcal{A}_n$ is the comultiplication map corresponding to the Frobenius system defining the $2$-dimensional TQFT.

\begin{figure}[ht]
\[
\begin{picture}(30,30)
\raisebox{3pt}{\scalebox{0.25}{\includegraphics{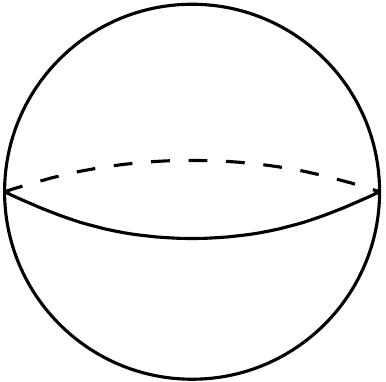}}}
\put(-15,21){\fontsize{9}{10.8}$i$}
\end{picture} 
\raisebox{12pt}{\ \ = \ $\begin{cases} 1,\ \  i = n - 1 \\ 0, \  \ \text{otherwise} \end{cases}$ \quad (S$_n$)  }
\hspace{1.5cm}
 \raisebox{-25pt}{\includegraphics[height = 0.85in]{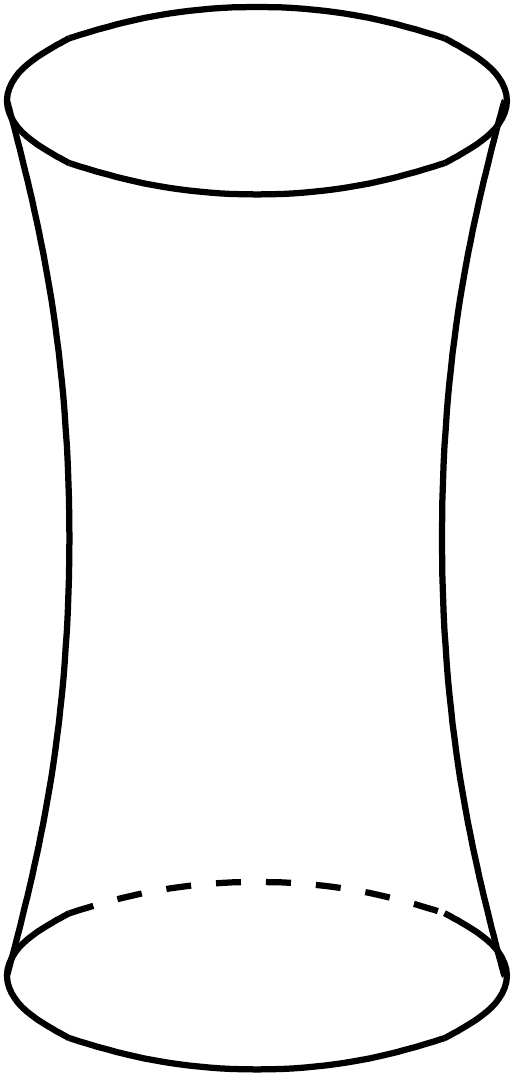}}  =   \sum_{i = 0}^{n-1} \raisebox{-25pt}{\includegraphics[height = 0.85in]{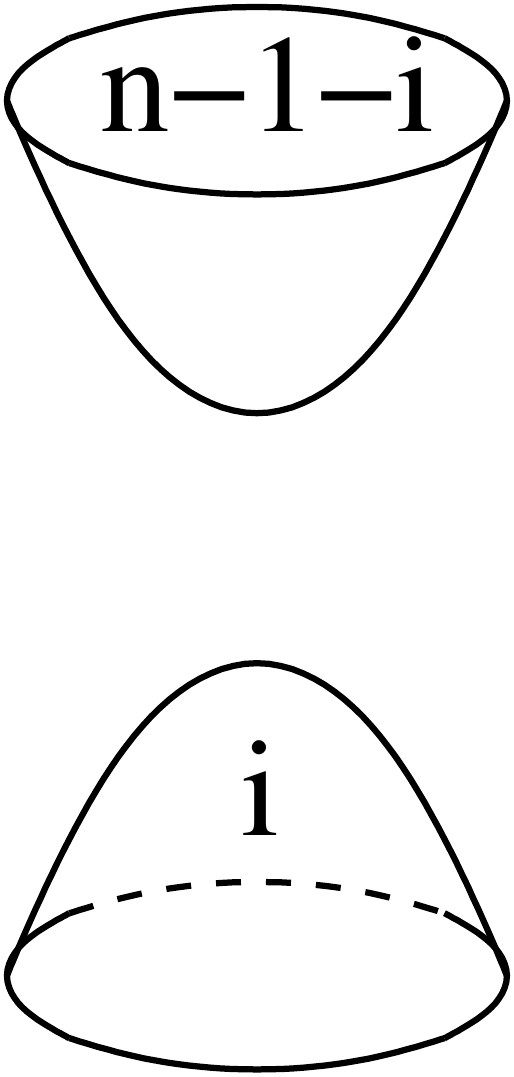}}
 \quad (\text{CN}_n) 
\]
\caption{The local relations (S$_n$) and (CN$_n$). }
\label{local relations}
\end{figure}

We denote by $\mathcal{T}_n$, for $n \geq 2$, the tautological functors in the above homology theories.
Recall that these functors are multiplicative with respect to disjoint unions of objects, as well as with respect to disjoint unions of morphisms, in the geometric categories $\textbf{Foam}_n$, for $n \geq 2$. 
It was proved in~\cite{Caprau09} that the universal $\mathfrak{sl}(2)$ homology theory satisfies the functoriality property with respect to link (and tangle) cobordisms without sign ambiguity. Clark~\cite{Cl} also proved that Khovanov's $\mathfrak{sl}(3)$ homology theory is properly functorial. Moreover, it was explained in~\cite{MV07} that the universal $\mathfrak{sl}(3)$ homology theory is functorial at least up to a minus sign (that is, up to multiplication by a unit in $\mathbb{Z}$). Finally, recall that the $\mathfrak{sl}(n)$ homology theory, for $n \geq 4$, is functorial (at least) up to multiplication by a non-zero complex number, as shown in~\cite{MSV}. Note that for the purpose of this paper, it suffices that a certain $\mathfrak{sl}(n)$ foam homology theory is functorial up to multiplication by a unit in the ground ring.

For the remainder of this section, embedded link cobordisms in $\mathbb{R}^3 \times [0, 1]$ may possibly be decorated with dots.

\begin{lemma}\label{spheres-cob}
Let $F \subset \mathbb{R}^3 \times [0, 1]$ be an embedded cobordism from a link $L_0$ to a link $L_1$.
Let $S$ be an unknotted $2$-sphere in $\mathbb{R}^3 \times [0, 1]$ and unlinked from $F$, and denote by $S^{(k)}$ the sphere $S$ marked with $k$ dots. Then, 
\begin{enumerate}
\item[(a)] $\mathcal{H}_2(F \cup S) = 0$ and $\mathcal{H}_2(F \cup S^{(1)}) = \mathcal{H}_2(F)$.
\item[(b)] $\mathcal{H}_3(F \cup S) = 0 = \mathcal{H}_3(F \cup S^{(1)})$ and $\mathcal{H}_3(F \cup S^{(2)}) = - \mathcal{H}_3(F)$.
\item[(c)] $\mathcal{H}_n(F \cup S) = \mathcal{H}_n(F \cup S^{(1)}) = \dots = \mathcal{H}_n(F \cup S^{(n-2)}) =  0$, and 

\hspace{-0.45cm}$\mathcal{H}_n(F \cup S^{(n-1)}) = \mathcal{H}_n(F)$, where $n \geq 4$.
\end{enumerate}
\end{lemma}

\begin{proof}
If necessary, we may perform an ambient isotopy of $\mathbb{R}^3 \times [0, 1]$ so that the unknotted $2$-sphere $S$ lies in a slice $\mathbb{R}^3 \times \{t\}$, for some $t \in [0,1]$, and that the intersection of $F$ with $\mathbb{R}^3 \times \{t\}$ is a $(1+1)$-cobordism. Then, the result in part (a) follows from the sphere relations (S$_2$) and the properties of the functors $\mathcal{T}_2$ and $\mathcal{H}_2$. Similarly, the sphere relations (S$_3$) and separately (S$_n$), for $n \geq 4$, together with the application of the functors $\mathcal{T}_3$ and $\mathcal{T}_n$, for $n \geq 4$, (along with the fact that $\mathcal{H}_3$ and $\mathcal{H}_n$ are functors) yield the equalities in parts (b) and (c).
 \end{proof}

\begin{lemma}\label{tubes-cobs}

Let $F \subset \R^3 \times [0,1]$ be an embedded cobordism from a link $L_0 \subset \R^3 \times \{0\}$ to a link $L_1 \subset \R^3 \times \{1\}$. Let $\gamma$ be a smoothly embedded arc with endpoints on $F$ and otherwise disjoint from $F$, and let $T$ be the boundary of an embedded tubular neighborhood of $\gamma$ (that is, $T$ is an annulus). Let $F'$ be the result of removing the neighborhood of $\partial \gamma$ from $F$ and attaching $T$.
Denote by $F^{(i, j)}$ the cobordism obtained from $F'$ by surgery  along a non-trivial compressing disk of $T$.

Then,
\begin{enumerate}
\item[(a)] $\mathcal{H}_2(F') = \mathcal{H}_2(F^{(1, 0)}) + \mathcal{H}_2(F^{(0, 1)}) -h \, \mathcal{H}_2(F)$.
\item[(b)] $ - \mathcal{H}_3(F') = \mathcal{H}_3(F^{(2, 0)}) + \mathcal{H}_3(F^{(1, 1)}) +  \mathcal{H}_3(F^{(0, 2)})$

\hspace{1.3cm}$ - a [\mathcal{H}_3(F^{(1, 0)}) + \mathcal{H}_3(F^{(0, 1)})] - b \mathcal{H}_3(F)$.

\item [(c)] $\mathcal{H}_n(F') = \displaystyle \sum_{i = 0}^{n-1} \mathcal{H}_n(F^{(i, n- 1 -i)})$, where $n \geq 4$.
\end{enumerate}
\end{lemma}

\begin{proof}
The proof is similar to that of Lemma~\ref{spheres-cob}, only that now we make use of the cutting neck relations. We perform first an isotopy of $\mathbb{R}^3 \times [0, 1]$ so that $T$ lies in a small ball contained in a slice $\mathbb{R}^3 \times \{t\}$, for some $t \in [0,1]$, and the intersections of $F'$ and $F^{(i, j)}$ with the ball can be identified with the pictures depicted in the cutting neck relations. The cutting neck relations imply that the morphisms in $\textbf{Foam}_n$ corresponding to the cobordisms in the statement of the lemma (where $n = 2$ in part (a), $n = 3$ in part (b), and $n \geq 4$ in part (c)) satisfy the skein relations in the statement. Then, the claimed identities on the homology groups follow at once from these, and from the properties of the tautological functors $\mathcal{T}_n$, and since $\mathcal{H}_n$ is a functor, for each $n \geq 2$.
\end{proof}

\begin{proposition}\label{cutting annulus}
Let $D \subset \R^3 \times [0,1]$ be an embedded cobordism from a link $L_0 \subset \R^3 \times \{0\}$ to a link $L_1 \subset \R^3 \times \{1\}$. Suppose $S$ is an unknotted $2$-sphere in $\R^3 \times [0,1]$ and unlinked from $D$.
Let $\gamma$ be a smoothly embedded arc with one endpoint on $F$ and the other on $S$, and otherwise disjoint from $D \cup S$, and let $T$ be the boundary of an embedded tubular neighborhood of $\gamma$ (that is, $T$ is an annulus). Let $D'$ be the result of removing the neighborhood of $\partial \gamma$ from $D \cup S$ and attaching $T$.

Then $\mathcal{H}_n(D') = \mathcal{H}_n(D)$, for all $n \geq 2$.
\end{proposition}

\begin{proof}
The proof follows from Lemmas~\ref{spheres-cob} and~\ref{tubes-cobs}. We apply first Lemma~\ref{tubes-cobs} to the cobordism $F: = D \cup S$, with $F': = D'$. Then note that $F^{(i, j)} = D^{(i)} \cup S^{(j)}$, where $D^{(i)}$ is the cobordism $D$ marked with $i$ dots, and $S^{(j)}$ is the 2-sphere $S$ marked with $j$ dots. So, we have
\begin{eqnarray*}
\mathcal{H}_2(D') &=& \mathcal{H}_2(D^{(1)} \cup S) + \mathcal{H}_2(D \cup S^{(1)}) -h \, \mathcal{H}_2(D \cup S)\\
&=& 0 + \mathcal{H}_2(D) -h \cdot 0\\
&=& \mathcal{H}_2(D) ,
\end{eqnarray*}
where the second equality holds due to part (a) in Lemma~\ref{spheres-cob}.
Similarly, using part (b) from Lemma~\ref{tubes-cobs}, we get,
\begin{eqnarray*}
 - \mathcal{H}_3(D') &=& \mathcal{H}_3(D^{(2)} \cup S) + \mathcal{H}_3(D^{(1)} \cup S^{(1)}) +  \mathcal{H}_3(D \cup S^{(2)})\\
 &&- a [\mathcal{H}_3(D^{(1)} \cup S) + \mathcal{H}_3(D \cup S^{(1)})] - b \mathcal{H}_3(D \cup S).
\end{eqnarray*}
Using part (b) from Lemma~\ref{spheres-cob}, we see that only the third term above, $\mathcal{H}_3(D \cup S^{(2)})$,  survives and equals to $-\mathcal{H}_3(D)$. Hence, $ \mathcal{H}_3(D') = \mathcal{H}_3(D)$, as desired.

Moreover, the following equalities follow from parts (c) of the previous two lemmas:
\begin{eqnarray*}
\mathcal{H}_n(D') =  \sum_{i = 0}^{n-1} \mathcal{H}_n(D^{(i)} \cup S^{(n- 1 - i)})
=\mathcal{H}_n(D \cup S^{(n-1)}) + 0= \mathcal{H}_n(D).
\end{eqnarray*}
Hence, the statement holds for every $n \geq 2$.
\end{proof}

We are now ready to prove the main result of this section.

\begin{theorem} \label{induced injective map}
Let $C$ be a ribbon concordance from a link $L_0$ to a link $L_1$. Then the induced maps on $\mathfrak{sl}(n)$ homologies
\[ \mathcal{H}_n (C) :  \mathcal{H}_n (L_0) \to  \mathcal{H}_n (L_1) \]
are injective, for all $n \geq 2$. 
\end{theorem}

\begin{proof}
Let $C$ be a ribbon concordance from $L_0$ to $L_1$, and let $\overline{C}$ be the mirror image of $C$ (that is, $\overline{C}$ is the reverse concordance from $L_1$ to $L_0$). Let $D: = \overline{C}  \circ C $. Then $D$ is a concordance from $L_0$ to itself. Since for each $n \geq 2$, the foam homology theory $\mathcal{H}_n$ is a functor, we have that:
\[ \mathcal{H}_n (D) = \mathcal{H}_n (\overline{C}) \circ \mathcal{H}_n(C), \, \text{for each} \, n \geq 2. \]

By the discussion at the beginning of this section, we know that the concordance $D$ can be obtained by taking the identity concordance $L_0 \times [0, 1]$ and ``tubing in'' unknotted, unlinked $2$-spheres $S_1, \dots, S_n$ using embedded annuli $T_1, \dots, T_n$. These annuli are the boundaries of embedded 3-dimensional 1-handles $h_1, \dots, h_i$ in $\mathbb{R}^3 \times[0, 1]$, where each $h_i$ connects $L_0 \times [0, 1]$ with $S_i$ and is disjoint from $S_j$, for $j \neq i$. Then, by Proposition~\ref{cutting annulus} and the functoriality properties of the corresponding foam homology theories, we get:

\[\mathcal{H}_2(D) = \mathcal{H}_2(L_0 \times [0, 1]) =  \text{id}_{\mathcal{H}_2(L_0)},\]
\[\mathcal{H}_3(D) = \pm \mathcal{H}_3(L_0 \times [0, 1]) = \pm \text{id}_{\mathcal{H}_3(L_0)}, \,\, \text{and}\]
\[\mathcal{H}_n(D) = q \, \mathcal{H}_n(L_0 \times [0, 1]) = q \, \text{id}_{\mathcal{H}_n(L_0)}, \,\, \text{for all} \,\, n \geq 4,\]
where $q \in \mathbb{Q}^*$.
 Therefore,
\[ \mathcal{H}_2(\overline{C}) \circ \mathcal{H}_2(C) =  \text{id}_{\mathcal{H}_2(L_0)},  \,\,  \mathcal{H}_3(\overline{C}) \circ \mathcal{H}_3(C) = \pm  \text{id}_{\mathcal{H}_3(L_0)}, \, \, \text{and} \]
\[ \mathcal{H}_n(\overline{C}) \circ \mathcal{H}_n(C) = q \,   \text{id}_{\mathcal{H}_n(L_0)}, \,\, \text{for some} \, q \in \mathbb{Q}^*.\]
In all of the above cases, the composition $\mathcal{H}_n(\overline{C}) \circ \mathcal{H}_n(C)$ is a bijective function, for each $n \geq 2$. 
Hence for each $n \geq 2$, $\mathcal{H}_n(C)$ is an injective map and $\mathcal{H}_n(\overline{C})$ is surjective. 
\end{proof}

As a consequence of Theorem~\ref{induced injective map}, we obtain that the homology theories $\mathcal{H}_n$, for all $n \geq 2$, give obstructions to ribbon concordance. For any concordance $C$ between links and any $n \geq 2$, the map $\mathcal{H}_n(C)$ preserves both the quantum and homological grading. Then the proof of the theorem implies that for any bigrading $(i, j)$ and $n \geq 2$, $\mathcal{H}_n^{i, j}(L_0)$ embeds in $\mathcal{H}_n^{i, j}(L_1)$ as a direct summand. 

\section{Gordian distance and spectral sequences in Khovanov homology}
\label{section:Spectral}

Lee \cite{Lee:Endo} defined an endomorphism of the Khovanov homology of a knot with coefficients in $\mathbb{Q}$, and Rasmussen \cite{Ras10} showed that Lee's endomorphism gives rise to a spectral sequence, called the Lee spectral sequence, whose $E_1$ page is isomorphic to the Khovanov homology of the knot. Shumakovitch \cite{Shumakovitch:Torsion} defined a version of Lee's spectral sequence with coefficients in the finite field $\mathbb{F}_p$ of order $p$, for an odd prime $p$. We refer to the above spectral sequences as the Lee spectral sequence with $R$ coefficients, where $R$ is either $\mathbb{Q}$ or $\mathbb{F}_p$ for an odd prime $p$.  A spectral sequence \textit{collapses} at the $k$th page if $E_{k-1}\neq E_k$ and $E_k=E_m$ for all $m\geq k$. When $R=\mathbb{Q}$ or $\mathbb{F}_p$, define $\pLee(K;R)$ to be the page at which the Lee spectral sequence with $R$ coefficients collapses. Similarly, Bar-Natan \cite{BN05} defined a variant of Khovanov homology with $\mathbb{F}_2$ coefficients. Turner \cite{Turner:BN} showed that Bar-Natan's variant gives rise to a spectral sequence similar in spirit to the Lee spectral sequence. Define $\pBN(K)$ to be the page at which the Bar-Natan spectral sequence collapses.

The \textit{Gordian distance} $d(K_1,K_2)$ between two knots $K_1$ and $K_2$ is the minimum number of crossing changes necessary to transform $K_1$ into $K_2$. The most famous Gordian distance is the \textit{unknotting number} $u(K)$ of a knot $K$, which is the Gordian distance between $K$ and the unknot. Kawauchi \cite{Kawauchi:Alternation} similarly defined the \textit{alternation number} $\alt(K)$ of a link $K$ to be the minimum Gordian distance between $K$ and the set of alternating knots. The Khovanov homology $Kh(K;R)$ of a knot over $R$ is \textit{homologically thin} if there is an integer $s$ such that $Kh^{i,j}(K;R)=0$, for $j-2i\neq s\pm 1$; that is, $Kh(K;R)$ is homologically thin if $Kh(K;R)$ is supported entirely in two adjacent diagonals $j-2i = s\pm1$. Define $\dthin(K;R)$ to be the minimum Gordian distance between $K$ and the set of knots that have thin Khovanov homology over $R$. Because every alternating link has thin Khovanov homology over $R$, for all rings $R$ that we consider, it follows that $\dthin(K;R)\leq \alt(K)$.

This section is organized as follows: the results in Subsection \ref{s41} are followed by examples in Subsection \ref{s42}, which illuminate the proofs provided in Subsection \ref{s43}. 

\subsection{Results}\label{s41}

For any real number $x$, define $\lceil x \rceil$ to be the ceiling of $x$; that is, $\lceil x \rceil$ is the least integer that is greater than or equal to $x$.  The next two results relate $\dthin(K;R)$ and $\alt(K)$ with the pages $\pLee(K;R)$, $\pBN(K)$ at which the Lee and Bar-Natan spectral sequences collapse.

\begin{theorem}
\label{theorem:LeeMain}
Let $K$ be a knot, and let $R$ be $\mathbb{Q}$ or $\mathbb{F}_p$, where $p$ is an odd prime. Then
\begin{equation}\label{ineq:LeeBound}
\pLee(K;R) \leq  \left\lceil \frac{\dthin(K;R)+3}{2}\right\rceil \leq \left\lceil \frac{\alt(K)+3}{2}\right\rceil.
\end{equation}
\end{theorem}
\begin{theorem}
\label{theorem:BNMain}
Let $K$ be a knot. Then
\begin{equation}\label{T3}
\pBN(K)\leq \dthin(K;\mathbb{F}_2)+2\leq \alt(K)+2.
\end{equation}

\end{theorem}

The Turaev genus of a knot is an invariant that measures how far a knot is from being alternating in a different way than the alternation number, and it is defined as follows. Each knot diagram $D$ has a Turaev surface of genus 
\[g_T(D)=\frac{1}{2}(2+c(D)-s_A(D)-s_B(D)),\]
where $c(D)$ is the number of crossings in $D$, and $s_A(D)$ and $s_B(D)$ are the number of components in the all-$A$ and, respectively, all-$B$ Kauffman states of $D$. 
The \textit{Turaev genus} $g_T(K)$ of a knot $K$ is defined as follows:
\[g_T(K) = \min \{g_T(D)~|~D~\textnormal{is a diagram of}~K\}.\]
It is known that a knot is alternating if and only if its Turaev genus is zero \cite{Turaev:Simple}.   The next result is a version of Theorems \ref{theorem:LeeMain} and \ref{theorem:BNMain}.

\begin{theorem}
\label{theorem:Turaev}
Let $R=\mathbb{Q}$ or $\mathbb{F}_p$ for an odd prime $p$. For any knot $K$,
\[
2\pLee(K;R) \leq  g_T(K)+4 
\qquad \textnormal{and}
\qquad \pBN(K) \leq  g_T(K) + 2.
\]
\end{theorem}
There are knots with arbitrarily large Turaev genus and alternation number one \cite{Lowrance:AltDist}. Also, there are knots with Turaev genus one that are conjectured to have arbitrarily large alternation number, and the existence of such knots would show that Theorem \ref{theorem:Turaev} does not immediately follow from Theorems \ref{theorem:LeeMain} and \ref{theorem:BNMain}.

We first give examples of how Theorems \ref{theorem:LeeMain} and \ref{theorem:BNMain} can be used, and then we prove each result.

\subsection{Examples}\label{s42}

Either side of the inequalities in Theorems \ref{theorem:LeeMain} and \ref{theorem:BNMain} can provide insight into the other. Example \ref{example:AH} gives a family of knots all of whose alternation numbers are one, but whose Khovanov homology becomes more and more complicated in terms of width. Despite having complicated Khovanov homology, Theorem \ref{theorem:LeeMain} implies that the Lee spectral sequence for this family of knots collapses at or before the second page, and Theorem \ref{theorem:BNMain} implies that the Bar-Natan spectral sequence collapses at or before the third page.

We remark that Alishahi and Dowlin \cite{AD:Lee} proved that if the unknotting number of a (nontrivial) knot is one or two, then the Lee spectral sequence collapses at the second page. However, many knots in Example \ref{example:AH} have unknotting number greater than two, and thus the results from \cite{AD:Lee} cannot be used for those knots.

Examples \ref{example:MM}, \ref{example:56}, and \ref{example:78} describe knots where the page at which the relevant spectral sequence collapses gives a nontrivial lower bound on the alternation number of the knot. 

Before describing the examples in detail, we remind the reader of some of the properties of the Lee and Bar-Natan spectral sequences. The map on the $E_r$ page of the Lee spectral sequence increases the homological grading by one and the polynomial grading by $4r$. Similarly, the map on the $E_r$ page of the Bar-Natan spectral sequence increases the homological grading by one and the polynomial grading by $2r$.  Khovanov homology with $\mathbb{F}_2$ coefficients splits as a direct sum of two copies of the reduced Khovanov homology with $\mathbb{F}_2$ coefficients; that is, $Kh^{i,j}(K;\mathbb{F}_2) \cong \widetilde{Kh}^{i,j-1}(K;\mathbb{F}_2)\oplus\widetilde{Kh}^{i,j+1}(K;\mathbb{F}_2)$. The Bar-Natan spectral sequence has this same behavior of splitting into two copies; see \cite{Turner:BN} for details.

\begin{example}
\label{example:AH}
For any pair of positive integers $m$ and $n$, de los Angeles Hernandez \cite{Hernandez:Alternation} constructed the hyperbolic knot $K(m,n)$ whose diagram is depicted in Figure \ref{figure:K(m,n)} and whose alternation number is one. Therefore, Theorem \ref{theorem:LeeMain} implies that the Lee spectral sequence of $K(m,n)$ collapses at or before the second page, and Theorem \ref{theorem:BNMain} implies that the Bar-Natan spectral sequence of $K(m,n)$ collapses at or before the third page.

Moreover, the width of the Khovanov homology of $K(m,n)$, that is the fewest number of adjacent $j-2i$ diagonals supporting $Kh(K(m,n))$, is $n+2$ \cite[Lemma 3.2]{Hernandez:Alternation}. 

Recall that if the unknotting number of a (nontrivial) knot is one or two, then the Lee spectral sequence collapses at the second page \cite{AD:Lee}. If $n+2<m$, then one can see that $K(m,n)$ has unknotting number greater than two, as follows. Dasbach and Lowrance \cite{DL:Turaev} proved that the signature of a knot $K$ with diagram $D$ satisfies the inequality 
\[s_A(D) -c_+(D) - 1 \leq \sigma(K) \leq -s_B(D) + c_-(D) + 1,\]
where $s_A(D)$ and $s_B(D)$ are the number of components in the all-$A$ and all-$B$ Kauffman states, respectively, and $c_+(D)$ and $c_-(D)$ are the number of positive and negative crossings in $D$. Applying this inequality to the diagram of $K(m,n)$, we see that $-2m-2n\leq \sigma(K) \leq -2m +2n$. Because $|\sigma(K)| \leq 2u(K)$, if $n+2< m$, then $u(K(m,n))> 2$. Hence, the result in \cite{AD:Lee} cannot be used for knots $K(m,n)$ with $n+2< m$.
\end{example}

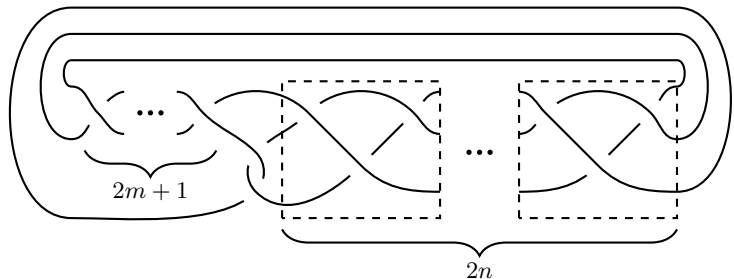
\begin{figure}
\[\begin{tikzpicture}[scale=.7, thick]

\def\gap{.2cm}

\coordinate (A) at (0,2);
\coordinate (A1) at (.5,2.4);
\coordinate (A2) at (.5,1.6);
\coordinate (B) at (2,2);
\coordinate (B1) at (1.5,2.4);
\coordinate (B2) at (1.5,1.6);
\coordinate (C) at (3,1.25);
\coordinate (D) at (3,.5); 
\coordinate (E) at (4,2);
\coordinate (F) at (5,1);
\coordinate (G) at (6,2);
\coordinate (H) at (8.5,2);
\coordinate (I) at (9.5,1);
\coordinate (J) at (10.5,2);
\coordinate (L1) at (-.5,0);
\coordinate (L2) at (-.5,1.5);
\coordinate (L3) at (-.5,2.5);
\coordinate (L4) at (-.5,3);
\coordinate (L5) at (-.5,3.5);
\coordinate (L6) at (-.5,4);
\coordinate (R1) at (11,.5);
\coordinate (R2) at (11,1.5);
\coordinate (R3) at (11,2.5);
\coordinate (R4) at (11,3);
\coordinate (R5) at (11,3.5);
\coordinate (R6) at (11,4);
\coordinate (M1) at (6.5,.5);
\coordinate (M2) at (6.5,1.6);
\coordinate (M3) at (6.5,2.4);
\coordinate (M4) at (8,.5);
\coordinate (M5) at (8,1.6);
\coordinate (M6) at (8,2.4);

\fill (1,2) circle (.05cm);
\fill (.8,2) circle (.05cm);
\fill (1.2,2) circle (.05cm);

\fill (7.25,1.25) circle (.05cm);
\fill (7.05,1.25) circle (.05cm);
\fill (7.45,1.25) circle (.05cm);

\draw[dashed] (3.5,0) rectangle (6.5,2.6);
\draw[dashed, xshift = 4.5cm] (3.5,0) rectangle (6.5,2.6);

\draw [
    thick,
    decoration={
        brace,
        mirror,
        amplitude = 10pt,
        raise=0.5cm
    },
    decorate
] (-.25,2) -- (2.25,2);
\draw (1,.5) node{\footnotesize{$2m+1$}};

\draw [
    thick,
    decoration={
        brace,
        mirror,
        amplitude = 10pt,
        raise=0.5cm
    },
    decorate
] (3.5,.5) -- (11,.5);
\draw (7.25,-1) node{\footnotesize{$2n$}};

\begin{scope}
	\begin{pgfinterruptboundingbox} 
		
		\path [invclip] (I) circle (\gap);
	\end{pgfinterruptboundingbox}
	\draw[thick] (I) to [out = 225, in = 0] (M4); 
	\end{scope}

\begin{scope}
	\begin{pgfinterruptboundingbox} 
		
		\path [invclip] (G) circle (\gap);
	\end{pgfinterruptboundingbox}
	\draw[thick] (G) to [out = 45, in = 180] (M3); 
	\end{scope}

\begin{scope}
	\begin{pgfinterruptboundingbox} 
		
		\path [invclip] (F) circle (\gap);
		\path [invclip] (G) circle (\gap);
	\end{pgfinterruptboundingbox}
	\draw[thick] (F) to [out = 45, in = 225] (G); 
	\end{scope}
	
\begin{scope}
	\begin{pgfinterruptboundingbox} 
		
		\path [invclip] (E) circle (\gap);
	\end{pgfinterruptboundingbox}
	\draw[thick] (E) to [out = 45, in = 135] (G) to [out = -45, in = 180] (M2); 
	\end{scope}

\begin{scope}
	\begin{pgfinterruptboundingbox} 
		
		\path [invclip] (C) circle (\gap);
		\path [invclip] (E) circle (\gap);
	\end{pgfinterruptboundingbox}
	\draw[thick] (C) to [out = 45, in = 225] (E); 
	\end{scope}

\begin{scope}
	\begin{pgfinterruptboundingbox} 
		
		\path [invclip] (C) circle (\gap);
		\path [invclip] (F) circle (\gap);
	\end{pgfinterruptboundingbox}
	\draw[thick] (C) to [out = 225, in = 135] (D) to [out =-45, in =225](F);
	\end{scope}

\begin{scope}
	\begin{pgfinterruptboundingbox} 
		
		\path [invclip] (B) circle (\gap);
	\end{pgfinterruptboundingbox}
	\draw[thick] (B) to [out = 45, in = 135] (E) to [out = -45, in =135] (F) to [out = -45, in=180] (M1); 
	\end{scope}

\begin{scope}
	\begin{pgfinterruptboundingbox} 
		
		\path [invclip] (B) circle (\gap);
	\end{pgfinterruptboundingbox}
	\draw[thick] (B2) to [out = 0, in = 225] (B); 
	\end{scope}

\begin{scope}
	\begin{pgfinterruptboundingbox} 
		
		\path [invclip] (D) circle (\gap);
	\end{pgfinterruptboundingbox}
	\draw[thick] (D) to [out = 45, in = -45] (C) to [out = 135, in = -45] (B) to [out = 135, in =0] (B1);
	\end{scope}

\begin{scope}
	\begin{pgfinterruptboundingbox} 
		
		\path [invclip] (D) circle (\gap);
	\end{pgfinterruptboundingbox}
	\draw[thick] (M6) to [out = 0, in = 135] (H) to [out = -45, in = 135] (I) to [out = -45, in = 180] (R1) to [out= 0, in = 0] (R6) to [out = 180, in = 0] (L6) to [out = 180, in =180] (L1) to [out = 0, in = 225, looseness = .7] (D);
\end{scope}

\begin{scope}
	\begin{pgfinterruptboundingbox} 
		
		\path [invclip] (J) circle (\gap);
		\path [invclip] (I) circle (\gap);
	\end{pgfinterruptboundingbox}
	\draw[thick] (J) to [out = 225, in = 45] (I);
\end{scope}

\begin{scope}
	\begin{pgfinterruptboundingbox} 
		
		\path [invclip] (H) circle (\gap);
	\end{pgfinterruptboundingbox}
	\draw[thick] (H) to [out = 225, in = 0] (M5);
\end{scope}

\begin{scope}
	\begin{pgfinterruptboundingbox} 
		\path [invclip] (A) circle (\gap);
		\path [invclip] (H) circle (\gap);
	\end{pgfinterruptboundingbox}
	\draw[thick] (A) to [out = 225, in = 0] (L2) to [out = 180, in = 180] (L5) to [out = 0, in = 180] (R5) to [out = 0, in = 0] (R2) to [out = 180, in =-45] (J) to [out =135, in = 45, looseness = 1] (H);	
\end{scope}

\begin{scope}
	\begin{pgfinterruptboundingbox} 
		\path [invclip] (J) circle (\gap);
	\end{pgfinterruptboundingbox}
	\draw[thick] (A2) to [out = 180, in = -45, looseness = 1] (A) to [out = 135, in = 0] (L3) to [out = 180, in = 180] (L4) to [out = 0, in = 180] (R4) to [out=0, in = 0] (R3) to [out = 180, in = 45] (J);	
\end{scope}

\begin{scope}
	\begin{pgfinterruptboundingbox} 
		\path [invclip] (A) circle (\gap);
	\end{pgfinterruptboundingbox}
	\draw
	[thick] (A1) to [out = 180, in = 225, looseness = 1] (A);	
\end{scope}

\end{tikzpicture}\]
\caption{A diagram of the knot $K(m,n)$.}
\label{figure:K(m,n)}
\end{figure} 

In Examples \ref{example:MM}, \ref{example:56}, and \ref{example:78}, we show the Khovanov homology of certain knots. The number in the $(i,j)$ entry of the table in Figure \ref{figure:MM}  is the rank of $Kh^{i,j}(K;R)$. All Khovanov homology computations for these examples are obtained using the program ``JavaKh-v2" available on the Knot Atlas \cite{knotatlas}.
\begin{figure}
\includegraphics[scale=.3]{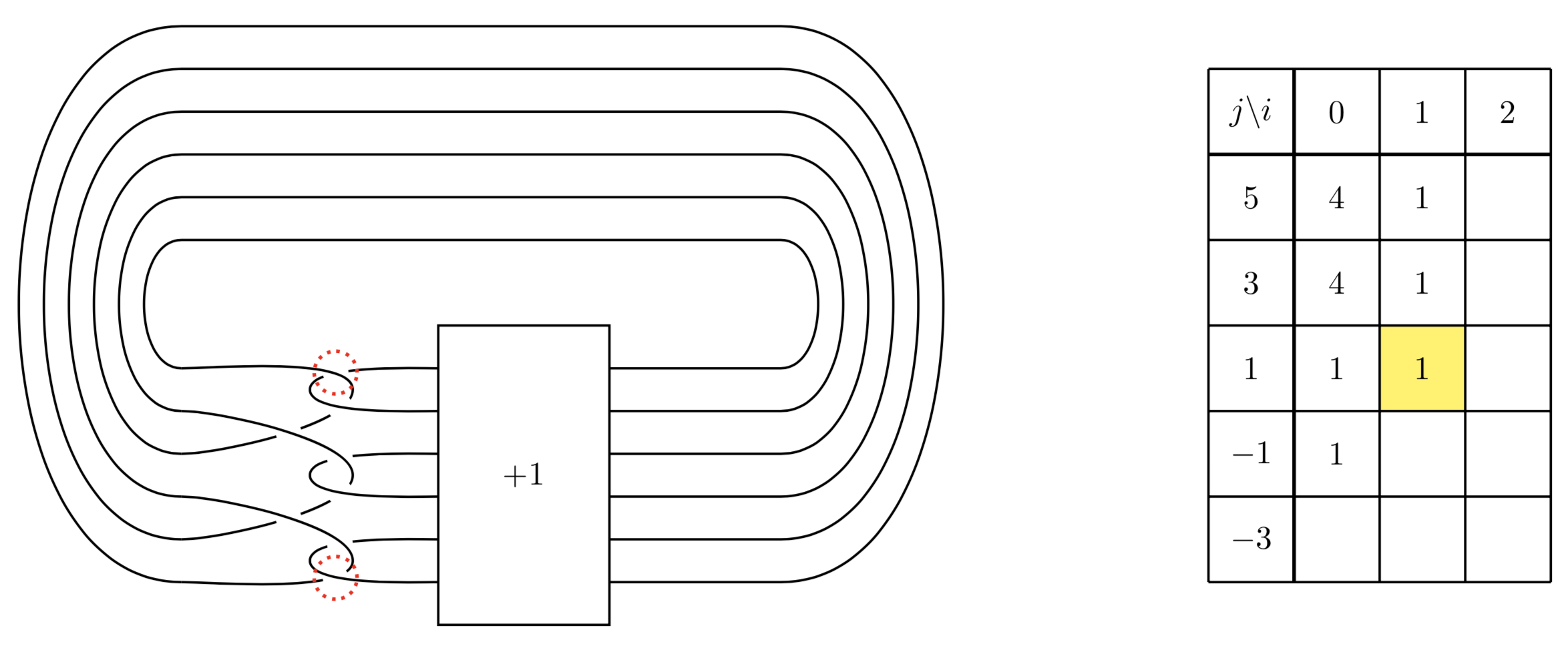}
\caption{The knot on the left has a positive full twist in the rectangle labeled $+1$. A portion of its Khovanov homology with $\mathbb{Q}$ coefficients is on the right. The highlighted yellow generator survives to the third page of the spectral sequence but not to the $E_\infty$ page.}
\label{figure:MM}
\end{figure}
\begin{example} 
\label{example:MM}
Manolescu and Marengon \cite{MM:Knight} gave an example of a knot $K$ whose Lee spectral sequence over $\mathbb{Q}$ does not collapse at the second page. This knot $K$ and a portion of its Khovanov homology $Kh(K,\mathbb{Q})$ appear in Figure \ref{figure:MM}. Because $Kh^{1,1}(K;\mathbb{Q})$ is nontrivial, while $Kh^{0,-3}(K;\mathbb{Q})$ and $Kh^{2,5}(K;\mathbb{Q})$ are trivial, it follows that $\pLee(K;\mathbb{Q})>2$. Changing the two crossings of $K$ circled in Figure \ref{figure:MM} transforms the knot into the figure-eight knot, and thus $\alt(K)\leq2$. Using now Theorem \ref{theorem:LeeMain}, it follows that $\alt(K)=2$. 
\end{example}

\begin{example}
\label{example:56}
The Lee spectral sequence for the $(5,6)$-torus knot $T_{5,6}$ with $\mathbb{Q}$ coefficients collapses at the second page; however, this is not the case when the coefficients are $\mathbb{F}_3$. Table \ref{figure:T56} shows the Khovanov homology of $T_{5,6}$ with $\mathbb{F}_3$ coefficients. Because $Kh^{13,43}(T_{5,6};\mathbb{F}_3)$ is nontrivial while $Kh^{12,39}(T_{5,6};\mathbb{F}_3)$ and $Kh^{14,47}(T_{5,6};\mathbb{F}_3)$ are trivial, it follows that $\pLee(T_{5,6};\mathbb{F}_3)>2$. Theorem \ref{theorem:LeeMain} implies that $2\leq \dthin(T_{5,6};\mathbb{F}_3) \leq \alt(T_{5,6})$. 
\end{example}

\begin{table}
\includegraphics[width=0.5\textwidth]{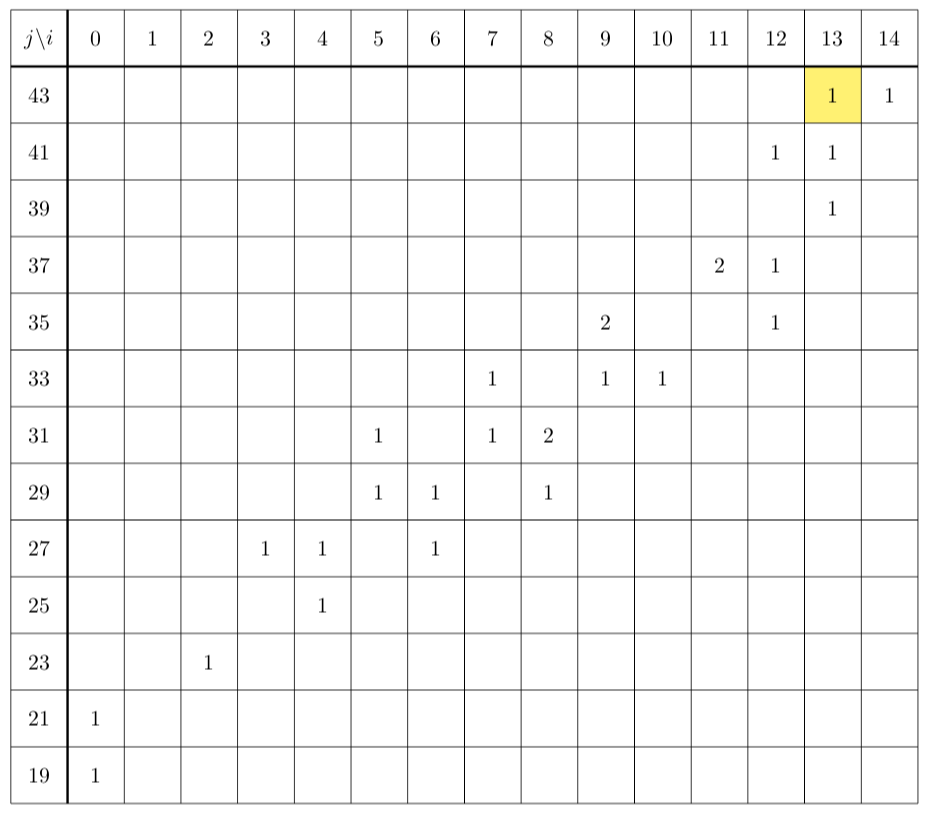}
\caption{The Khovanov homology of $T_{5,6}$ with $\mathbb{F}_3$ coefficients. The highlighted yellow generator survives to the third page of the spectral sequence, but not to the $E_\infty$ page.}
\label{figure:T56}
\end{table}

\begin{example}
\label{example:78}
The Khovanov homology of $T_{7,8}$ with $\mathbb{F}_2$ coefficients is shown in Table \ref{figure:T78}. Since $i=26$ is the maximum homological grading where $Kh^{i,j}(T_{7,8};\mathbb{F}_2)$ is nontrivial, the summands $Kh^{26,79}(T_{7,8};\mathbb{F}_2)$ and $Kh^{26,81}(T_{7,8};\mathbb{F}_2)$ must be paired with the summands $Kh^{25,75}(T_{7,8};\mathbb{F}_2)$ and $Kh^{25,77}(T_{7,8};\mathbb{F}_2)$ on the third page of Bar-Natan spectral sequence. Consequently, the summands $Kh^{25,79}(T_{7,8};\mathbb{F}_2)$ and $Kh^{25,81}(T_{7,8};\mathbb{F}_2)$ must be paired with the summands $Kh^{24,71}(T_{7,8};\mathbb{F}_2)$ and $Kh^{24,73}(T_{7,8};\mathbb{F}_2)$ on the fourth page of the Bar-Natan spectral sequence. Therefore $\pBN(T_{7,8})\geq 4$, and thus Theorem \ref{theorem:BNMain} implies that $2\leq \dthin(T_{7,8};\mathbb{F}_2)\leq \alt(T_{7,8})$.
\end{example}

\begin{table}
\includegraphics[width=0.8\textwidth]{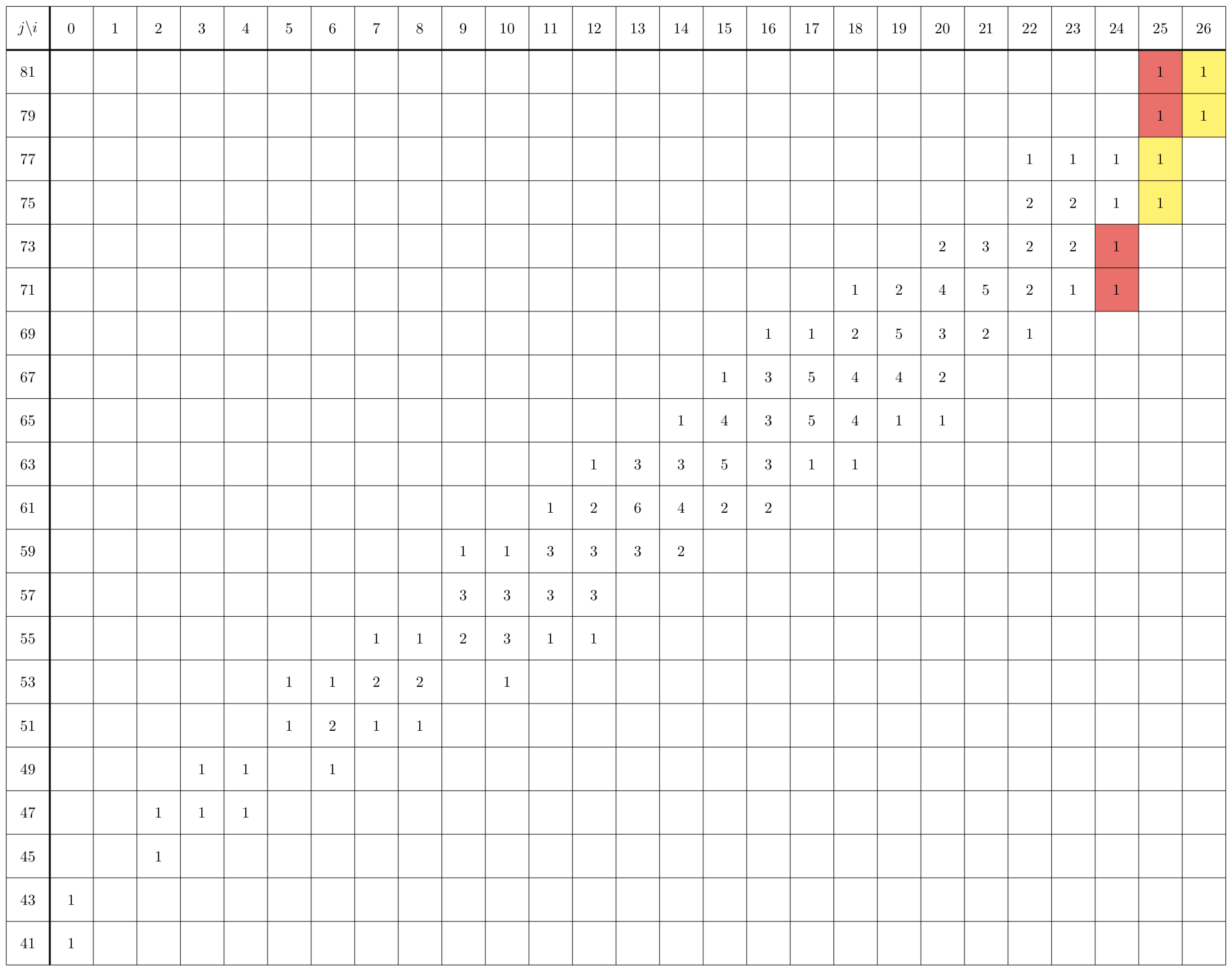}
\caption{The Khovanov homology $Kh(T_{7,8};\mathbb{F}_2)$ of $T_{7,8}$. The highlighted yellow generators survive to the $E_3$ page, and the highlighted red generators survive to the $E_4$ page of the Bar-Natan spectral sequence.}
\label{figure:T78}
\end{table}

\subsection{Proofs} \label{s43} The Lee and Bar-Natan spectral sequences both arise as spectral sequences of filtered complexes. The filtration comes from adding the Khovanov differential to different boundary maps that increase the polynomial/quantum grading. 
The Lee and Bar-Natan spectral sequences arise from maps $\dlee:CKh^{i,j}(D;R)\to CKh^{i+1,j+4}(D;R)$ and $\dBN:CKh^{i,j}(D;\mathbb{F}_2) \to CKh^{i+1,j+2}(D;\mathbb{F}_2)$, respectively. For any knot diagram $D$, the homology of $(CKh(D;R),d+\dlee)$ is isomorphic to $R\oplus R$ situated in homological grading zero, and similarly, the homology of $(CKh(D;\mathbb{F}_2),d+\dBN)$ is isomorphic to $\mathbb{F}_2\oplus\mathbb{F}_2$ situated in homological grading zero.

Bar-Natan \cite{BN05} constructed a deformation of Khovanov homology using coefficients in $\mathbb{F}_2[h]$ for a formal variable $h$ instead of $\mathbb{F}_2$ and using the differential $d+h\dBN$ instead of the usual Khovanov differential $d$. Turner later viewed the Bar-Natan construction through the lens of spectral sequences as described above. Alishahi and Dowlin \cite{AD:Lee} similarly encapsulated the Lee endomorphism as part of a deformed complex with coefficients in $\mathbb{Q}[X,t]/(X^2=t)$, where the differential in this complex is $d+ t \dlee$. Just as with Lee's endomorphism, one can replace $\mathbb{Q}$ with $\mathbb{F}_p$, for any odd $p$, and all of the results of \cite{AD:Lee} hold without changing their proofs.

An element $\alpha$ in the homology of Bar-Natan's complex is \textit{$h$-torsion of order $n$} if $h^n\alpha=0$ but $h^{n-1}\alpha\neq 0$. Let $\mathfrak{u}_h(K)$ be the maximum order of any torsion element in the homology of Bar-Natan's complex. Then $\mathfrak{u}_h(K) + 1 = \pBN(K)$ \cite[Lemma 3.2]{Alishahi:BN}. 

Similarly, an element $\alpha$ in the deformed  Lee homology over $R=\mathbb{Q}$ or $\mathbb{F}_p$, for an odd prime $p$, is \textit{$X$-torsion of order $n$} (respectively \textit{$t$-torsion of order $m$}) if $X^n\alpha = 0$ but $X^{n-1}\alpha\neq 0$ (respectively $t^m\alpha=0$ but $t^{m-1}\alpha\neq 0$).  Alishahi and Dowlin proved the following facts about $\mathfrak{u}_X(K;\mathbb{Q})$ and $\mathfrak{u}_t(K;\mathbb{Q})$. 
We observed that the proofs of these facts when $R = \mathbb{Q}$ also apply when using $\mathbb{F}_p$ coefficients. As such, we state the following for $R=\mathbb{Q}$ or $\mathbb{F}_p$, where $p$ is an odd prime.
\begin{enumerate}
\item If $Kh(K;R)$ is homologically thin, then $\mathfrak{u}_X(K;R)=1$;
\item $|\mathfrak{u}_X(K_+;R)-\mathfrak{u}_X(K_-;R)|\leq 1$, where $K_+$ and $K_-$ are knots differing by a single crossing change;
\item $\lceil \mathfrak{u}_X(K;R)/2\rceil = \mathfrak{u}_t(K;R)$, and
\item $\mathfrak{u}_t(K)+1 = \pLee(K;R)$.
\end{enumerate}

We are now in a position to prove Theorems \ref{theorem:LeeMain}, \ref{theorem:BNMain}, and \ref{theorem:Turaev}.

\begin{proof}[Proof of Theorem \ref{theorem:LeeMain}]
Let $\dthin(K;R)=d$. Hence, there is a sequence of knots $K=K_0, K_1, \dots, K_d$ such that $K_{i+1}$ is obtained from $K_i$ via a crossing change for all $i=0,\dots, d-1$, and $Kh(K_d;R)$ is homologically thin. Item (1) above implies that $\mathfrak{u}_X(K_d;R)=1$, and item (2) implies that $\mathfrak{u}_X(K;R)\leq d+1$. Then item (3) implies that $\mathfrak{u}_t(K;R) =\left \lceil \frac{\mathfrak{u}_X(K;R)}{2}\right\rceil\leq \left \lceil \frac{d+1}{2}\right\rceil.$ Finally, item (4) implies that $\pLee(K;R) = \mathfrak{u}_t(K;R)+1 \leq \left\lceil\frac{d+3}{2}\right\rceil$, as desired. The second inequality in the theorem follows at once from the fact that $\dthin(K;R) \leq \alt(K)$, as seen in the beginning of this section.
\end{proof}

\begin{proof}[Proof of Theorem \ref{theorem:BNMain}]
Let $\dthin(K;R)=d$. Hence there is a sequence of knots $K=K_0, K_1, \dots, K_d$ such that $K_{i+1}$ is obtained from $K_i$ via a crossing change for $i=0,\dots, d-1$, and $Kh(K_d;\mathbb{F}_2)$ is homologically thin. By Alishahi \cite{Alishahi:BN}, since $K_i$ and $K_{i+1}$ differ by a crossing change, it follows that $|\mathfrak{u}_h(K_i) - \mathfrak{u}_h(K_{i+1})|\leq 1$, and thus $\mathfrak{u}_h(K)\leq d+ \mathfrak{u}_h(K_d)$. Since $Kh(K_d;\mathbb{F}_2)$ is homologically thin, $\pBN(K_d) \leq 2$.  But $\pBN(K_d)=\mathfrak{u}_h(K_d)+1$, and thus $\mathfrak{u}_h(K_d)\leq 1$.  It follows that $\mathfrak{u}_h(K)\leq d+ 1$, and therefore $\pBN(K)\leq d+2$, as desired.
\end{proof}

\begin{proof}[Proof of Theorem \ref{theorem:Turaev}]
The \textit{width} $w(Kh(K;R))$ of the Khovanov homology over a ring $R$ is defined as
\[w(Kh(K;R)) = 1 + \frac{1}{2}\left(\max\{j-2i~|~Kh^{i,j}(K;R)\neq 0\} - \min\{j-2i~|~Kh^{i,j}(K;R)\neq 0\}\right).\]
Champanerkar, Kofman, and Stoltzfus \cite{CKS:Turaev} proved that $w(Kh(K;R))\leq g_T(K)+2$. Since the Lee differential on the $E_r$ page increases the homological grading $i$ by one and the polynomial grading $j$ by $4r$, if $\pLee(K;R)=n$, then $w(Kh(K;R))\geq2n-2$. Therefore $2\pLee(K;R) \leq g_T(K)+4$, as desired. Similarly, since the Bar-Natan differential on the $E_r$ page increases the homological grading by one and the polynomial grading by $2r$, if $\pBN(K)=n$, then $w(Kh(K;\mathbb{F}_2))\geq n$. Therefore, $\pBN(K)\leq g_T(K)+2$.
\end{proof}

\bibliographystyle{plain}
\bibliography{references}

\end{document}